\documentclass[12pt, fleqn, reqno]{amsart}
\usepackage{newtxtext}
\usepackage[frenchmath]{newtxmath}
\usepackage{eucal}
\usepackage{mathrsfs}
\usepackage{tikz-cd}
\usetikzlibrary{decorations.pathmorphing}
\usepackage{microtype}
\usepackage{hyperref}
\usepackage{nccmath}
\usepackage{setspace}
\usepackage{titlesec}%

\titleformat{\section}{\normalsize\bfseries}{\thesection.}{1em}{}
\titlespacing*{\section}{0pt}{11pt plus 2pt minus 2pt}{11pt plus 2pt minus 2pt}
\titlelabel{\thetitle.}

\usepackage[letterpaper,portrait,left=1.375in,right=1.375in]{geometry}

\newcommand{\ra}{\rightarrow}

\newcommand{\CC}{\mathbf{C}}
\newcommand{\ZZ}{\mathbf{Z}}
\newcommand{\QQ}{\mathbf{Q}}

\newcommand{\PP}{\mathbf{P}}

\newcommand{\A}{\mathbf{A}}

\newcommand{\Lotimes}{\otimes^{\mathbf{L}}}

\usepackage{enumitem,moreenum}
\usepackage{amsmath,amsthm}
\usepackage{mathtools}
\usepackage{extarrows}
\usepackage[utf8]{inputenc}
\usepackage{accents}
\makeatletter
\renewenvironment{proof}[1][\proofname]{\par
  \normalfont
  \topsep6\p@\@plus6\p@ \trivlist
  \item[\hskip\labelsep\itshape
    #1\@addpunct{.\enspace\textemdash}]\ignorespaces
}{%
  \qed\endtrivlist
}

\def\@secnumfont{\bfseries}
\makeatother

\newtheoremstyle{bfth}
{}
{}
{\itshape}
{\parindent}
{\itshape}
{\textbf{.}\hspace{.5em}\textemdash}
{.5em}
{{\thmname{#1}\upshape\thmnumber{ \textbf{(#2)}}\thmnote{ (\textit{#3})}}}

\newtheoremstyle{bfthstar}
{}
{}
{\itshape}
{\parindent}
{\itshape}
{.\hspace{.5em}\textemdash}
{.5em}
{{\thmname{#1}\upshape\thmnumber{ \textbf{(#2)}}\thmnote{ (\textit{#3})}}}

\newtheoremstyle{bfrm}
{}
{}
{}
{\parindent}
{\itshape}
{\textbf{.}\hspace{.5em}\textemdash}
{.5em}
{{\thmname{#1}\upshape\thmnumber{ \textbf{(#2)}}\thmnote{ (\textit{#3})}}}

\theoremstyle{bfth}
\newtheorem{theorem}{Theorem}[section]

\newtheorem{lemma}[theorem]{Lemma}
\newtheorem{proposition}[theorem]{Proposition}
\newtheorem{proposition-definition}[theorem]{Proposition-Definition}

\theoremstyle{bfrm}
\newtheorem{remark}[theorem]{Remark}
\newtheorem{remarks}[theorem]{Remarks}

\theoremstyle{bfthstar}
\newtheorem*{theorem*}{Theorem}
\newtheorem*{corollary*}{Corollary}
\newtheorem*{lemma*}{Lemma}
\newtheorem*{proposition*}{Proposition}
\newtheorem*{proposition-definition*}{Proposition-Definition}

\theoremstyle{definition}

\newtheorem*{definition*}{Definition}

\DeclareMathOperator{\id}{id}

\DeclareMathOperator{\Fr}{Fr}
\DeclareMathOperator{\GL}{GL}

\newcommand{\lra}{\longrightarrow}

\newcommand{\minus}{\smallsetminus}

\long\def\comment#1{}

\makeatletter
\newcommand*\dotp{{\mathpalette\dotp@{.5}}}
\newcommand*\dotp@[2]{\mathbin{\vcenter{\hbox{\scalebox{#2}{$\m@th#1\bullet$}}}}}
\newcommand{\leqnomode}{\tagsleft@true}
\newcommand{\reqnomode}{\tagsleft@false}
\makeatother

\title{The derived category of the abelian category of constructible sheaves}
\author{Owen Barrett}
\date{May 21, 2020}

\begin{document}
\thispagestyle{plain}
\maketitle

\section{Introduction}\label{sec:intro}
Let $X/k$ be an algebraic variety (i.e. a separated scheme of finite type)
over a base field $k$. Let $p$ be the characteristic of $k$ and $\ell$ a
prime number, $\ell\ne p$.
Denote by $D(X)=D(X,\Lambda)$ the usual triangulated category of bounded
constructible complexes of $\Lambda$-sheaves on $X$;
here the coefficient ring $\Lambda$ is either finite with $\ell$ nilpotent
in it, or, in the $\ell$-adic setting, $\Lambda$ is either a finite 
extension $E$ of $\QQ_\ell$ or its ring of integers $R_E$.
\begin{theorem}\label{th:derived}
	Suppose $k$ is algebraically closed. Then $D(X,\Lambda)$ is equivalent
	to the bounded derived category of the abelian category of
	constructible $\Lambda$-sheaves.
\end{theorem}
For $\Lambda$ finite, the theorem is valid for any $k$ and follows 
from~\cite[VI 5.8, IX 2.9\,(iii)]{SGA4}.
In the $\ell$-adic setting it was proved in Nori's remarkable article~\cite{Nori} under the assumption that $p=0$.
The aim of this note is to remove that assumption.
\begin{remarks}
\begin{enumerate}[label=(\roman*)]
\item	In his paper, Nori embeds $k$ into $\CC$ and considers the
corresponding Betti version of constructible sheaves with arbitrary 
coefficient rings $R$.
If $R=R_E$ then these are the same as étale $R_E$-sheaves.
\item Nori’s theorem 3\,(b) asserts that~\eqref{th:derived} above is true when
$R$ is a field. But his theorem 3\,(a) implies that~\eqref{th:derived} is true
for $R=R_E$ as well since torsion free constructible sheaves generate
$D(X,R_E)$ as a triangulated 
category.\footnote{To see this, note that every constructible sheaf is an
extension of a torsion free sheaf by a torsion sheaf, and every torsion 
sheaf admits a finite filtration with successive quotients of the form 
$j_!\mathcal F$, where $\mathcal F$ is a locally constant torsion sheaf 
on a locally closed subvariety $j: Y\hookrightarrow X$. It remains to find a 
surjection $\alpha:\mathcal G\twoheadrightarrow j_!\mathcal F$ where
$\mathcal G$ is torsion free
(its kernel $\mathcal K$ is torsion free then as well, and
$j_!\mathcal F=\operatorname{Cone}(\mathcal K\ra\mathcal G)$).
Let $p: T\ra Y$ be a finite etale covering such that $p^*\mathcal F$ is
constant. Pick a surjection $(R_E)^m_T\twoheadrightarrow p^*\mathcal F$.
The promised $\alpha$ is the composition
$j_!p_*(R_E)^m_T\twoheadrightarrow j_!p_*p^*\mathcal F\twoheadrightarrow j_!\mathcal F$.}
\end{enumerate}
\end{remarks}
Nori deduces theorem~\ref{th:derived} from the following
theorem~\ref{th:NoriAn} by an argument that works for arbitrary $p$.
\begin{theorem}\label{th:NoriAn}
	Suppose $k$ is algebraically closed.
	Then every constructible $\Lambda$-sheaf on $\A^n$
	is a subsheaf of a constructible sheaf $\mathcal G$ with
	$H^q(\A^n,\mathcal G)=0$ for $q>0$.
\end{theorem}
It is in the proof of theorem~\ref{th:NoriAn} that Nori uses the assumption
$p=0$. Precisely, the proof goes by induction on $n$.
Given a constructible sheaf $\mathcal F$ on $\A^n$, we can assume that the
projection $\A^n\to\A^{n-1}, (x_1,..,x_n)\mapsto (x_1,..,x_{n-1})$, becomes
finite on the singular locus $Y$ of $\mathcal F$.
The key fact used by Nori is that then $\mathcal F$ is equisingular along
$x_n=\infty$: the reason for this is that if $k=\CC$, so that we can use the
classical topology, then for any ball $U\subset\A^{n-1}$ and punctured disc
$D^\circ$ in $\A^1$ around $\infty$, a local system on $U\times D^\circ$ is
the same as a local system on $D^\circ$.
This assertion is in no sense valid in case $p>0$ due to the effect of wild
ramification. We show that, by a slight elucidation of Achinger’s result
\cite[3.6]{Achinger}, the above equisingularity is achieved after turning
$\mathcal F$ by an appropriate quadratic automorphism of $\A^n$.
This is enough to make the rest of Nori’s argument work for $p>0$.

I wish to thank P.\;Achinger, M.\;Nori and T.\;Saito for valuable
suggestions and advice;
I especially wish to thank A.\;Beilinson for years of kind instruction.

\section{On a theorem of Achinger}\label{sec:Achinger}
We explain a version of Achinger's theorem~\cite[3.6]{Achinger}.
The constructions of the proof are Achinger's original ones;
we streamline his argument (avoiding the use of a theorem of Deligne-Laumon)
and obtain a precise description of the singular support.\footnote{For the notion of singular support, see Beilinson~\cite{Beilinson} and
Saito~\cite{Saito}.}
In this section, the coefficient ring $\Lambda$ is finite.

Let $\pi:\A^{n+1}\ra\A^n,\overline\pi:\A^n\times\PP^1\ra\A^n,j:\A^{n+1}\hookrightarrow\A^n\times\PP^1$
be the evident projections and open embedding;
set $D:=\A^n\times\{\infty\}=\A^n\times\PP^1\minus\A^{n+1}$.
Let $\mathcal F\in D(\A^{n+1})$ be a constructible complex and $Y$ a closed
subset of $\A^{n+1}$, $Y\neq\A^{n+1}$, such that $\mathcal F$ is locally
constant\footnote{i.e. all cohomology sheaves of $\mathcal F$ are locally
constant.} on $\A^{n+1}\minus Y$.
\begin{theorem}\label{th:AchingerAn}
	If $k$ is infinite, then one can find an automorphism $g$ of $\A^{n+1}$ 
	such that
	\begin{enumerate}[label=(\roman*)]
		\item the restriction of $\pi$ to $g^{-1}(Y)$ is finite, and
		\item the restriction of $SS(j_!g^*\mathcal F)$ to the complement of
		$g^{-1}(Y)$ is either empty (if $\mathcal F$ is supported on $Y$) or
		is the union of the zero section and the conormal to $D$.
	\end{enumerate}
	If $k$ is finite, then one can find $g$ as above after a finite extension 
	of $k$.
\end{theorem}
\begin{proof}
Let $x_1,\ldots,x_n$ and $s$ be the linear coordinates on $\A^n$ and
$\A^1$, and $s:t$ be the homogeneous coordinates on $\PP^1$, so that
$\overline\pi(x_1,\ldots,x_n,(s:t))=(x_1,\ldots,x_n)$ and 
$j(x_1,\ldots,x_n,s)=(x_1,\ldots,x_n,(s:1))$. Let
$(y_0:y_1:\ldots:y_n:y_{n+1})$ be the homogeneous coordinates on $\PP^{n+1}$
and $j':\A^{n+1}\hookrightarrow\PP^{n+1}$ be the embedding
$(x_1,\ldots,x_n,s)\mapsto(1:x_1:\ldots:x_n:s)$; let $\overline Y$ be
the closure of $Y$ in $\PP^{n+1}$.
If $\mathcal F$ is supported on $Y$, we can take for $g$ any linear
transformation such that $g(0:\ldots:0:1)\not\in\overline Y$.
So we assume that $\mathcal F$ is generically nonzero.
Our $g$ will be of the form $g=ah$, where $a$ is a linear automorphism of
$\A^{n+1}$, and $h$ is the automorphism
$(x_1,\ldots,x_n,s)\mapsto(x_1+s^2,x_2,\ldots,x_n,s)$.
Notice that $h$ extends to a map
\begin{equation*}
	\overline h:\A^n\times\PP^1\ra\PP^{n+1}\qquad
	(x_1,\ldots,x_n,(s,t))\mapsto(t^2:t^2x_1+s^2:t^2x_2:\ldots t^2x_n:st)
\end{equation*}
sending $D$ to $c:=(0:1:0:\ldots:0)\in\PP^{n+1}$.
Let $W$ be a vector space with coordinates $z,x_2,\ldots,x_n,s$ and $\sigma$
be the map
\begin{equation*}
	\sigma:W\ra\PP^{n+1}\minus(y_1=0)\qquad
	(z,x_2,\ldots,x_n,s)\mapsto(z^2:1:x_2:\ldots:x_n:s);
\end{equation*}
note that $\sigma(0)=c$.

\textsc{Step 1.} We formulate conditions (a)–(c) on $\mathcal F$ and prove that they
assure that our theorem holds with $g=h$. The cases $p\neq 2$ and $p=2$ are 
considered separately.

\emph{Case} $p\neq 2$: Our conditions are
(a) $c\not\in\overline Y$,
(b) the fiber $SS(\sigma^*j'_!\mathcal F)_0$ has dimension 1,
i.e. $SS(\sigma^*j'_!\mathcal F)_0$ is the union of a finite nonempty set of 
lines, and
(c) $\tau:=(1,0,\ldots,0,1)\in T_0W=W$ does not lie in the union of 
hyperplanes orthogonal to the lines from (b).

Condition (a) implies that $h^{-1}(Y)=\overline h\,^{-1}(\overline Y)$ 
is closed in $\A^n\times\PP^1\ra\A^n$, so that (i) of 
theorem~\ref{th:AchingerAn} holds. Let us check (ii).
We need to show that the restriction of
$SS(j_!h^*\mathcal F)$ to $D$ is the conormal to $D$. It is enough to
check our claim replacing $\A^n\times\PP^1$ by an étale neighborhood $V$
of $D$; we choose $V$ to be an étale covering of
$\A^n\times\PP^1\minus((s=0)\cup(t^2x_1+s^2=0))$ obtained by adding
$v,v^2=1+(t/s)^2x_1$, to the ring of functions; the embedding
$D\hookrightarrow V$ is $t\mapsto0,v\mapsto1$. The restriction of
$\overline h$ to $V$ lifts to a map
\begin{equation*}
	\kappa:V\ra W\qquad
	(x_1,\ldots,x_n,(1,t),v)\mapsto
	(t/v,(t/v)^2x_2,\ldots,(t/v)^2x_n,t/v^2)
\end{equation*}
sending $D$ to $0\in W$ and the normal vector $\partial_t$ at any point
of $D$ to $\tau$. Now (b) and (c) mean that $\kappa$ is properly
$SS(\sigma^*j_!\mathcal F)$-transversal on a neighborhood of $D$, and so,
by Saito~\cite[8.15]{Saito},\footnote{In the formulation of~\cite[8.15]{Saito}, Saito assumes that $k$ is perfect. This assumption is redundant since singular support is compatible with change of base field by~\cite[1.4\,(iii)]{Beilinson}, and so Saito's assertion amounts to the one for the base change of the data to the perfect closure of $k$.}
$SS(\kappa^*\sigma^*j_!'\mathcal F)=\kappa^\circ SS(\sigma^*j_!'\mathcal F)$,
which is the union of the zero section and the conormal to $D$.
We are done since $\kappa^*\sigma^*j_!'\mathcal F$ is the
pullback of $j_!h^*\mathcal F$ to $V$.

\emph{Case} $p=2$: Consider the purely inseparable map
$\Fr_1:\A^{n+1}\ra\A^{n+1},(x_1,\ldots,x_n,s)\mapsto(x_1^2,x_2,\ldots,x_n,s)$.
Set $F:=h\Fr_1 h^{-1}$ and $\mathcal F':=F_*\mathcal F$, so
$F^*\mathcal F'=\mathcal F$. Our conditions are
(a) $c\not\in\overline Y$,
(b) $\dim SS(\sigma^*j'_!\mathcal F')_0=1$, and
(c) $\tau$ does not lie in the union of hyperplanes orthogonal to the lines in $SS(\sigma^*j_!'\mathcal F')_0$.
Let us check that they imply that theorem~\ref{th:AchingerAn} holds with 
$g=h$. As above, (a) implies (i) of the theorem.
The map $\Fr_1$ extends to a map $\A^n\times\PP^1\ra\A^n\times\PP^1$ also
called $\Fr_1$; notice that $\overline h\Fr_1$ lifts to a map
$\chi:\A^n\times\PP^1\minus((s=0)\cup(tx_1+s=0))\ra W$ given in
coordinates by
\begin{equation*}
	(x_1,\ldots,x_n,(1,t))\mapsto
	\Big(\frac{t}{tx_1+1},\frac{t^2x_2}{t^2x_1^2+1},\ldots,\frac{t^2x_n}{t^2x_1^2+1},\frac{t}{t^2x_1^2+1}\Big).
\end{equation*}
As before, $\chi(D)=0$ and $\chi$ is properly
$SS(\sigma^*j_!'\mathcal F')$-transversal, so
$SS(\chi^*\sigma^*j'_!\mathcal F')=\chi^\circ SS(\sigma^*j'_!\mathcal F')$ is 
again the union of the zero section and the conormal to $D$.
As $\Fr_1^*h^*\mathcal F'=h^{-1}_*\mathcal F=h^*\mathcal F$, we are done.

\textsc{Step 2.}
It remains to find a linear transformation $a$ of $\A^{n+1}$ such 
that $a^*\mathcal F$ satisfies the above conditions (a)–(c).
Let $G\subset\GL(n+1)$ be the group of linear transformations $a$ such that
$a^*(x_1)=x_1$. The action of $G$ on $\PP^{n+1}$ lifts to an action of $G$
on $W$ fixing the coordinate $z$.
Let $D'\subset W$ denote the hyperplane $(z=0)$.
The vector bundle $TW|_{D'}$ over $D'$ decomposes as the direct sum of the
tangent bundle to $D'$ and the normal line bundle $\ker(d\sigma)$;
the action of $G$ preserves this decomposition.
Let $\PP(TW)|^\circ_{D'}$ be the complement in $\PP(TW)|_{D'}$ of the
projectivization of these subbundles. One checks that the action of $G$ on
this open subset is transitive. So, since the point in $\PP(TW)$ that
corresponds to $\tau$ lies in $\PP(TW)^\circ_{D'}$, its $G$-orbit is
Zariski open in $\PP(TW)|_{D'}$.

Let $C\subset T^*W$ be $SS(\sigma^*j_!\mathcal F)$ if $p\neq2$ or
$SS(\sigma^*j_!\mathcal F')$ if $p=2$.
There is a nonempty open subset $Q$ of $D'$ such that every fiber of $C$
over $Q$ is the union of a finite nonempty set of lines.
Indeed, since $\dim C=n+1$ by~\cite{Beilinson}, one can find $Q$ such that
the fibers of $C$ over $Q$ have dimension $\leq1$; their dimension equals 1
since $\sigma^*j_!\mathcal F$, or $\sigma^*j_!\mathcal F'$, is not locally 
constant at the generic point of $D'$ (recall that $\mathcal F$ does not
vanish at the generic point of $\A^{n+1}$).
Shrinking $Q$, we may assume it does not intersect $\overline Y$.
By the above, we can find $a\in G(k)$ such that $a(c)\in Q$ and $a(\tau)$
lies in the complement of the hyperplanes orthogonal to the lines in
the fiber of $C$ at $a(c)$, at least when $k$ is infinite.
When $k$ is finite, we may need to pass to a finite extension of $k$ to find
the promised $a$.
\end{proof}
\begin{remarks}\begin{enumerate}[label=(\roman*)]
\item The assertion of theorem~\ref{th:AchingerAn} and its proof remain 
valid in the setting of holonomic $\mathcal D$-modules.
(The reference to~\cite[8.15]{Saito} should be replaced by
theorem 4.7\,(3) in \cite[\S4.4]{Kashiwara}.)
This answers positively Kontsevich's question from~\cite[\S1.3]{Achinger}.
\item Achinger's theorem~\cite[3.6]{Achinger} asserts that in the case of
locally constant $\mathcal F$, $\overline\pi$ is locally acyclic rel.
$j_!\mathcal F$. This follows immediately from theorem~\ref{th:AchingerAn}.
\end{enumerate}	
\end{remarks}

\section{Adapting Nori's argument to characteristic $p$}
We explain how to modify Nori's proof of the theorems from 
\S\ref{sec:intro} to make it work in any characteristic.
In fact, only one step of the proof of theorem~\ref{th:NoriAn}, which goes
by induction on $n$, requires modification. We present it as
proposition~\ref{prop:Nori2.2} below; it replaces Nori's proposition 2.2.

As in \S\ref{sec:intro}, the coefficient ring $\Lambda$ is either finite or
$\ell$-adic. For a map $?$ between algebraic varieties, we write $?_*$
instead of $R?_*$, etc.

Suppose we are in the situation of \S\ref{sec:Achinger}; we follow the
notation there.
For a geometric point $s$ of $\A^n$ we denote by
$i_s:\A^1=\pi^{-1}(s)\hookrightarrow\A^{n+1},\overline i_s:\PP^1=\overline\pi\,^{-1}(s)\hookrightarrow\A^n\times\PP^1$ the embeddings 
of the $s$-fibers; let $j_s:\A^1\hookrightarrow\PP^1$ be the restriction of
$j$ to the $s$-fibers.
\begin{lemma}\label{lem:Nori1.3A}
	If $k$ is infinite, then
	one can find an automorphism $g$ of $\A^{n+1}$ such that $g^{-1}(Y)$ is
	finite over $\A^n$, and for every $s$ as above the base change morphism
	$\overline i_s^*j_*g^*\mathcal F\ra j_{s*}i_s^*g^*\mathcal F$ is an
	isomorphism.
	If $k$ is finite, then such $g$ exists after a finite extension of $k$.
\end{lemma}
\begin{proof}
	Assume for the moment that $\Lambda$ is finite.
	Our $g$ is the automorphism from theorem~\ref{th:AchingerAn} applied to 
	the Verdier dual $D\mathcal F$ of $\mathcal F$. We check our claim on 
	the complement of $g^{-1}(Y)$, which is a neighborhood of the divisor 
	$D$. By~\eqref{th:AchingerAn}, $\overline i_s$ is
	$SS(j_!g^*D\mathcal F)$-transversal there, and so one has
	\begin{multline*}
		\overline i_s^!(j_!g^*D\mathcal F)
		=\overline i_s^*(j_!g^*D\mathcal F)(-n)[-2n] \\
		=j_{s!}i_s^*(g^*D\mathcal F)(-n)[-2n]
		=j_{s!}D(i_s^*g^*\mathcal F)=D(j_{s*}i_s^*g^*\mathcal F);
	\end{multline*}
	here the first equality comes from~\cite[8.13]{Saito}.
	Applying $D$ to both sides, we are done.
	
	The case of $\ell$-adic coefficients follows from that of finite
	coefficients. Indeed, for $\mathcal F\in D(\A^{n+1},R_E)$ the assertion
	of the lemma for $\mathcal F\Lotimes\ZZ/\ell$ amounts to the one for
	$\mathcal F$ (with the same $g$ and $Y$),\footnote{To see this, notice
	that $\mathcal G\in D(X,R_E)$ vanishes iff $\mathcal G\Lotimes\ZZ/\ell=0$.
	Thus, a morphism $\mathcal K\ra\mathcal L$ in $D(X,R_E)$ is an 
	isomorphism iff
	$\mathcal K\Lotimes\ZZ/\ell\ra\mathcal L\Lotimes\ZZ/\ell$
	is an isomorphism (since being an isomorphism amounts to the vanishing
	of the cone). Now apply this remark to the morphism from the lemma.}
	and implies the one for $\mathcal F\otimes\QQ_\ell$.
\end{proof}
Henceforth $\mathcal F$ is a constructible sheaf (not a complex).
Replacing $\mathcal F$ by $g^*\mathcal F$, we assume 
lemma~\ref{lem:Nori1.3A} holds for $g=\id$. Enlarging $Y$ (so that it is
still finite over $\A^n$) we assume that $\pi(Y)=\A^n$.
Replacing $\mathcal F$ by its extension by zero from $\A^{n+1}\minus Y$,
we assume $\mathcal F|_Y=0$.
\begin{proposition}\label{prop:Nori2.2}
	Under the above conditions, one can find an embedding of constructible
	sheaves $\delta:\mathcal F\hookrightarrow\mathcal C$ such that
	$\pi_*\mathcal C=0$.
\end{proposition}
\begin{proof}
	Nori provides a natural construction of the sheaf $\mathcal C$ and the
	arrow $\delta$. Let $p_1,p_2:\A^{n+2}\ra\A^{n+1}$ denote the
	canonical projections.
	We define a sheaf $\mathcal B$ on $\A^{n+2}$ via the exact sequence
\begin{equation*}\label{eq:defB}\tag{$\dagger$}
	0\lra\mathcal B\lra p_1^*\mathcal F\lra\Delta_*\mathcal F\lra 0.
\end{equation*}
Applying $p_{2*}$ to this sequence, the long exact sequence of
cohomology gives the arrow
\begin{equation*}
	\delta:\mathcal F
	=p_{2*}\Delta_*\mathcal F\ra H^1p_{2*}\mathcal B=:\mathcal C.
\end{equation*}
Let us check the promised properties. One has
$p_{2*}p^*_1\mathcal F\xleftarrow\sim\pi^*\pi_*\mathcal F$ by smooth base
change. Therefore for every geometric point
$t$ of $\A^{n+1}$, putting $s:=\pi(t)$ one has
\begin{equation*}
	(p_{2*}p_1^*\mathcal F)_t=(\pi_*\mathcal F)_s
	=(\overline\pi_*j_*\mathcal F)_s=R\Gamma(\A^1,i^*_s\mathcal F),
\end{equation*}
the last equality following from proper base change and 
lemma~\ref{lem:Nori1.3A}. By Artin's theorem, the latter complex is acyclic
off degrees 0 and 1; as $i^*_s\mathcal F$ is the extension by 
zero of a locally constant sheaf from the complement of a nonempty finite
set $Y_s$, its degree 0 cohomology vanishes, and the complex
$p_{2*}p_1^*\mathcal F$ is in fact concentrated in degree 1.
This implies that $p_{2*}$ transforms~\eqref{eq:defB} into a distinguished
triangle whose long exact sequence of cohomology reduces to the short exact
sequence of sheaves
\begin{equation*}
	0\lra\mathcal F\xlongrightarrow\delta\mathcal C\lra H^1p_{2*}p_1^*\mathcal F\lra0,
\end{equation*}
showing $\delta$ injective and $p_{2*}\mathcal B$ acyclic off
degree 1, i.e. $\mathcal C=p_{2*}\mathcal B[1]$. If $\Pi:=\pi p_1=\pi p_2$,
\begin{equation*}
	\pi_*\mathcal C=\Pi_*\mathcal B[1]
	=\operatorname{Cone}(\Pi_*p_1^*\mathcal F\lra\Pi_*\Delta_*\mathcal F)=0,
\end{equation*}
and we have the proposition.
\end{proof}
\begin{remark}\label{rem:cd1}
	If $k$ is no longer supposed algebraically closed, but of cohomological
	dimension $\leq1$, the proof of theorem~\ref{th:NoriAn} gives that if
	$\mathcal F$ is a constructible sheaf on $\A^n$, there exists
	a monomorphism $\mathcal F\hookrightarrow\mathcal G$ into a
	constructible $\mathcal G$ inducing the null morphism
	$H^q(\A^n,\mathcal F)\ra H^q(\A^n,\mathcal G)$ for $q>0$, and
	theorem~\ref{th:derived} continues to hold over such a $k$.
\end{remark}

\end{document}